\documentclass{amsart}

\usepackage{amssymb}
\usepackage[stretch=10,shrink=10]{microtype}
\usepackage[shortlabels]{enumitem}
\usepackage{tikz}
\usetikzlibrary{arrows.meta}
\usepackage{comment}
\usepackage{hyperref}
\hypersetup{colorlinks,linkcolor={red},citecolor={olive},urlcolor={red}}
\usepackage{mathtools}
\mathtoolsset{showonlyrefs}

\newcommand{\f}{\frac}
\newcommand{\ind}[1]{\mathbf{1}{\left\{ #1 \right\}}}
\renewcommand{\phi}{\varphi}
\newcommand{\E}{\mathbf E}
\newcommand{\X}{X}
\newcommand{\x}{\xi}
\renewcommand{\a}{\mathbf{a}}
\newcommand{\A}{\mathbf{A}}
\newcommand{\B}{\mathcal{B}}
\renewcommand{\P}{\mathbf P}
\newcommand{\TT}{\mathbb T}

\DeclareMathOperator{\Poi}{Poi}
\DeclareMathOperator{\Geo}{Geo}

\DeclareMathOperator{\Bin}{Bin}
\DeclareMathOperator{\TFM}{TFM}

\DeclareMathOperator{\TSSFM}{SSTFM}

\usepackage{theoremref}

\newtheorem{theorem}{Theorem}
\newtheorem{lemma}[theorem]{Lemma}

\newtheorem{proposition}[theorem]{Proposition}

\theoremstyle{definition}

\theoremstyle{remark}

    \title[A stochastic combustion model with  thresholds on trees]{A stochastic combustion model with thresholds on trees}

\author[Junge]{Matthew Junge}
\email{Matthew.Junge@baruch.cuny.edu}

\author[McDonald]{Zoe McDonald}

\author[Pulla]{Jean Pulla}

\author[Reeves]{Lily Reeves}

\thanks{The authors received partial support from NSF grant DMS-1855516. 
Part of this research was completed during the 2021 Baruch College Discrete Math REU partially supported by NSF grant DMS-2051026. We are grateful to Tobias Johnson for his valuable input}

\begin{document}

\begin{abstract}
Place one active particle at the root of a graph and a Poisson-distributed number of dormant particles at the other vertices. Active particles perform simple random walk. Once the number of visits to a site reaches a random threshold, any dormant particles there become active. For this process on infinite $d$-ary trees, we show that the  total number of root visits undergoes a phase transition. 
%We additionally prove that this model on integer lattices has no phase transition. 
%Our findings generalize previous results for the usual frog model, which can be viewed as having threshold one at each site. 
\end{abstract}

\maketitle

\section{Introduction}

We study an interacting particle system in which active particles diffuse across a graph interspersed with dormant particles which become active once a visit threshold is met. There is initially one active particle at the root and $\X_v$ dormant particles at each nonroot site $v$. The $\X_v$ are independent with common probability measure $\x$ supported on the nonnegative integers. Each site is also independently assigned a threshold $T_v$ with common probability measure $\tau$ supported on the positive integers with $\tau(\infty)<1$. We will write $\X$ and $T$ for generic random variables corresponding to $\x$ and $\tau$, respectively. Active particles perform simple random walk in discrete time steps while dormant particles remain in place.  Once the total number of visits to $v$ by active particles reaches $T_v$, any dormant particles at $v$ are converted to active particles.

%\HOX{Perhaps we should add a clarification about notation to the begininning of this paragraph. Suggestion: "Given a graph $G$ with distinguished root vertex, a threshold measure $\tau$ and initial distribution $\x$, we denoted the above model by TFM($G,\tau,\X$)." -ZM}
%Natural choices for $\x$ include $\x = \delta_1$ i.e., one particle per site, and $\x \sim \Poi(\mu)$ so that there is a Poisson with mean $\mu$ number of dormant particles per site. Choices for $\tau$ to keep in mind are $\tau = \delta_k$ so that the threshold is a deterministic integer $k \geq 1$, and $\tau \sim \Geo(p)$ so that the threshold is a Geometric random variable supported on $1,2\hdots$ with success parameter $p$. Note we will sometimes write $\Geo_0$ to denote a Geometric distribution supported on $0,1,\hdots$.

These dynamics may be written as a reaction equation by viewing $A$-particles as active and $B_n$-particles as dormant at a site that has received $n$ total visits from $A$-particles. The reaction rules at $v \in G$ are then
\begin{align}
A+B_n \to \begin{cases} A+A,& n = T-1 \\ A+B_{n+1}, & \text{otherwise} \end{cases}.\label{eq:rule}
\end{align}
%\HOX{Should we write $T_v$ instead of $T$? -ZM}
Note that if multiple $A$-particles arrive simultaneously, each visit contributes towards meeting the threshold.

The dynamics at \eqref{eq:rule} with $\tau =\delta_1$ have been interpreted as a model for combustion as well as rumor/infection spread \cite{ramirez2004asymptotic, alves2002shape}. It is natural to introduce thresholds, since in these applications multiple interactions may be required to spark a reaction in the such systems. Other recent work used similar dynamics with thresholds to model viral phages cooperating to overcome immunity in host bacteria \cite{brouard2022invasion, landsberger2018anti}. 
From a mathematical perspective, it is interesting to measure the effect different thresholds have on the propagation of $A$-particles.

Because of the chaotic manner in which $B$-particles are converted to $A$-particles, the process with reactions at \eqref{eq:rule} is often colloquially referred to as the \emph{frog model}; $A$-particles are \emph{active frogs} and $B$-particles are \emph{sleeping frogs}. The vivid imagery, though removed from applications, is useful when describing the process. For this reason, we will refer to particles as frogs and call the process the \emph{threshold frog model on $G$ with threshold $\tau$ and initial configuration $\x$} which we abbreviate with $\TFM(G,\tau,\x)$. If $\tau$ or $\x$ are point masses $\delta_k$ or have a named distribution we will often replace them with the simpler representation. For example, $\TFM(G,\delta_1,\Poi(\mu))$  has $\tau=\delta_1$ and $\x \sim \Poi(\mu)$ where $\Poi(\mu)$ denotes a Poisson distribution with mean $\mu$.

%The model has received a lot of attention since the first paper studying it by Telcs and Wormald in \cite{telcs1999branching}.  
There is a significant body of work devoted to studying frog models on $\mathbb Z^d$. 
The first published result about the frog model came from Telcs and Wormald \cite{telcs1999branching} and concerned the number of visits to the root for $\TFM(\mathbb Z^d,\delta_1,\delta_1)$. Later, Alves, Machado, and Popov in \cite{alves2002shape} proved that the set of visited sites scaled linearly by time converges to a deterministic limiting shape. A similar result in continuous time was proven by Ram\'irez and Sidoravicius  in \cite{ramirez2004asymptotic}. 
Subsequent work has studied variations at the front \cite{comets2009fluctuations, berard2010large, kesten2012asymptotic}. Additionally, some variants have attracted recent attention including: the frog model with death \cite{alves2002phase, lebensztayn2019new}; competitive dynamics with two species of $A$-particles \cite{deijfen2019competing, roy2021coexistence};  a version in which $B$-particles are linked via clusters in critical bond percolation \cite{junge2020critical}; versions in which the random walk paths have drift \cite{gantert2009recurrence, dobler2018recurrence, beckman2019frog}; and an adaptation to Euclidean space \cite{beckman2018asymptotic}.
To our knowledge no previous work has considered the threshold variant described at \eqref{eq:rule}. 

%We initiated the study of it during the Summer 2021 Baruch College Discrete Math REU. It took us some time to realize that there is a natural way to compare a threshold frog model to a frog model with threshold one but a smaller initial density of $B$-particles. 
%Although this connection does not require a long proof, it helps explain the threshold frog model. We go onto describe some consequences and a few lingering questions.

%Unlike the $p=1$ case, the set of sites at which the threshold has been met is not connected. 

\subsection{Result}

Our focus is the threshold frog model  on the infinite complete $d$-ary tree $\mathbb T_d$ i.e., the rooted tree in which each vertex has $d$ child vertices. Denote the root by $\varnothing$. We assume that $\x \sim \Poi(\mu)$. 
To lighten notation we will write
$$\TFM_d(\tau, \mu) := \TFM(\mathbb T_d, \tau, \Poi(\mu)).$$

A fundamental statistic is the total number of visits by active frogs to the root up to time $t$, which we label $V_t$. It is convenient to not count the presence of the initially awake frog at the root as a visit. So, formally, if $A_s$ is the number of active frogs at $\varnothing$ at time $s$, then $V_t = \sum_{s=1}^t A_s$. Let $V = \lim_{t\to \infty} V_t$. We call a realization of the process \emph{recurrent} if $V= \infty$, and otherwise call it \emph{transient}. Explicitly, we are interested in the quantity 
$$\mu_c(\tau,d) := \inf \{ \mu \colon \TFM_d(\tau, \mu) \text{ is recurrent almost surely}\}.$$
%Note that $\mu_c(\tau,d)$ could be equivalently defined using a supremum since frog models with larger $\mu$ are more likely to be recurrent. Also, 
% \HOX{this is referring to just the last paragraph of the proof \cite[Proposition 1.4]{beckman2019frog}, right? ---TLJ} \HOX{Delete ``of''? -Lily} 
A simple adaptation of the last paragraph of the proof of \cite[Proposition 1.4]{beckman2019frog} gives that recurrence of $\TFM$ satisfies a $0$-$1$ law. Thus, it is equivalent to define $\mu_c(\tau,d)$ for recurrence occurring with positive probability.

Although trees are often simpler settings to study statistical physics models, this is not the case for the frog model. %Compare to the process on the integer lattice where random walk diffuses more slowly and so less accuracy is needed. 
The process on the integer lattice was proven by Popov to be recurrent for all $\mu>0$ \cite{popov2001frogs}. It appears to us that one could easily extend this result to prove that $\TFM(\mathbb Z^d, \tau, \x)$ is recurrent for all $d, \tau,$ and $\x$. 
On the other hand, it was unknown for over a decade whether or not there was a recurrent phase for the frog model with $\tau \sim \delta_1$ on a $d$-ary tree \cite{popov2003frogs, gantert2009recurrence}.  
Hoffman, Johnson, and Junge resolved that question by proving that the frog model with one particle per site and threshold one is recurrent on the binary tree but transient on the $d$-ary tree with $d \geq 5$\cite{hoffman2017recurrence}. The same authors showed that $0<\mu_c(\delta_1,d) < \infty$ \cite{hoffman2016transience}. So the frog model with threshold one and $\xi \sim \Poi(\mu)$ undergoes a phase transition from transience to recurrence as $\mu$ is increased. 

% the process on a $d$-ary tree undergoes a phase transition; there is a critical density $\mu_c(\delta_1,d)$ such that for $\mu< \mu_c(\delta_1,d)$ $\TFM_d(\delta_1,\Poi(\mu))$ is transient almost surely while for $\mu > \mu_c(\delta_1,d)$ the process is recurrent almost surely. 

Transience and recurrence behavior for the frog model on trees is subtle. 
Random walk paths inherit a drift from the tree structure. Individual particles visit each level of the tree a small number of times. 
Widespread activation is required to overcome the drift and persistently send active particles to the root. 
Thresholds introduce a new difficulty since active frogs will often silently pass through sites with sleeping frogs. It is not immediately clear whether or not recurrence can occur with thresholds.
Indeed, we were unable to find any argument that directly used the known recurrence of the frog model on trees without thresholds. Our main result establishes that an analogous phase transition occurs for any threshold. 

\begin{theorem}\thlabel{thm:main} For all $d \geq 2$ it holds that $0<\mu_c(\tau,d) < \infty.$
\end{theorem}

A corollary of \thref{thm:main} is that this transition holds for $\TFM(\mathbb T_d, \tau, \xi)$ with more general $\xi$. Establishing transience for small initial densities only uses a branching process comparison. So $\TFM(\mathbb T_d, \tau, \xi)$ can always be made transient by taking the mean of $\xi$ sufficiently small. As for recurrence, the comparison result from \cite{johnson2018stochastic} implies that $\TFM(\mathbb T_d, \tau, \xi)$ is recurrent so long as $\xi$ dominates a Poisson distribution with mean $m > \mu_c(\tau,d)$ in the increasing concave (icv) order. Roughly speaking, the icv order rewards distributions for being concentrated. For example, a frog model with a fixed number of particles $m > \mu_c(\tau,d)$ is more concentrated, and thus recurrent on $\mathbb T_d$. See \cite{johnson2018stochastic} for more details.

%Notice that the upper bound is tremendously large. For example, if $\tau \sim \delta_2$ so that the threshold is two at each site, we must take $\mu > 1.9*10^{32}$ to ensure recurrence. Compare to \cite{johnson2016critical} which proves that when $\tau \sim \delta_1$ the process is recurrent with $\mu >1.45$. While this may seem absurd, there is no obvious way to directly apply the result with threshold 1 to that with higher thresholds.  As discussed in Section~\ref{sec:overview}, It took significant innovation over past approaches

%Note that adding $A$ particles to the process results in a stochastically larger number of visits by $A$-particles to the root up to time $t$. Thus, for fixed $d$ we have $\mu_c(p,d)$ increases as $p \to 0$. 
%In \cite{johnson2016critical}, Johnson and Junge showed that $\mu_c(\delta_1,2) \leq 1.45$ and that $.24 d \leq \mu_c(\delta_1,d) \leq 2.28d$ for all large enough $d$. We use this to provide asymptotic bounds on $\mu_c(\tau,d)$. 

% Our second result moves away from $\mathbb T_2$ to the integer lattice $\mathbb Z^d$. While this is a natural, well-studied setting, it is less interesting from the perspective of transience and recurrence since there is no phase transition. 

% \begin{theorem} \thlabel{thm:Zd}
%  $\TFM(\mathbb Z^d, \tau, \X)$ is recurrent almost surely for any $\tau$ and $\x$. 
% \end{theorem}

\subsection{Proof overview} \label{sec:overview} 

 Our main contribution is the upper bound on $\mu_c(\tau, d)$. It is obtained by considering a modified version of $\TFM_d(\tau, \mu)$ that trims from the ranges of the random walk paths to produce an embedded sub-process with less root visits inspired by the \emph{self-similar frog model} introduced in \cite{hoffman2017recurrence}. We restrict active particles to follow a lazy non-backtracking portion of their full random walk path. To accommodate thresholds, we only allow the threshold $T_v$ at a site $v$ to be met if the first active particle that visits $v$ accrues at least $T_v-1$ lazy steps (plus its initial step) there.

In \thref{lem:rec}, we prove that the number of visits to the root in our restricted model $V'$ satisfies a recursive distributional equation (RDE). The RDE
involves $(d-1) N$ independent thinned copies of $V'$ with $N$ a geometric random variable.
Our approach to deriving and analyzing the RDE follows the same general strategy as introduced in \cite{hoffman2016transience}. Namely, we use it to prove that if $V' \succeq \Poi(\lambda)$, then in fact $V' \succeq \Poi(\lambda+1)$. See \thref{prop:dom}. Bootstrapping implies that $V' = \infty$ almost surely. As $V\succeq V'$, this gives \thref{thm:main}.

 The recursive equation satisfied by $V'$ is less homogeneous than past analogues. In particular, the number of i.i.d.\ copies of $V'$ is random and these variables are thinned by different (random) amounts. The most novel aspect of the proof of \thref{thm:main} is the derivation of an exact formula for the probability a given subset of variables are ``activated" in the Poisson version of our RDE. See \thref{lem:A}. The formula is a basic consequence of Poisson thinning (see \thref{lem:thinning}). %\HOX{Cite Lemma 12 here. Also we cite Lemma 12 in the proof of Lemmas 8 \& 9 but nowhere else. Maybe only cite the first time? -Lily} 
 However, past works \cite{hoffman2016transience, johnson2016critical} missed this characterization, and instead relied on simpler approximations. 
%We truly need the formula in \thref{lem:A} to carry out our argument. 
%It is not possible to use the approximations from \cite{hoffman2016transience, johnson2016critical} to prove an analogue of \thref{prop:dom}.  
The approaches from \cite{hoffman2016transience, johnson2016critical} are too crude to prove \thref{prop:dom}. 
\thref{lem:A} allows us to balance the Poisson quantity of root visits with proportional probability bounds that penalize few visits. This balancing act can be seen in the formula \eqref{eq:hexpand} and tractable bound at \eqref{eq:hc}. These, and \thref{lem:A}, are the most novel technical aspects of our work.

\subsection{Discussion and a further question}
\thref{thm:main} shows that, like the usual frog model, the threshold frog model on trees  has a transience to recurrence phase transition. This confirms that the models are qualitatively similar. However, we expect that the quantitative behavior of the threshold is rather different. It would be interesting to describe the asymptotic growth of $\mu_c(\tau,d)$ as $d \to \infty$. It is proven in \cite{johnson2016critical} that $\mu_c(\delta_1, d) = O(d)$.  Our approach can be refined to give a super-exponential bound $\log_d \mu_c(\tau,d) = O( d^{C}/\E[(d+1)^{-\tau} ])$. We believe that the true growth rate is much smaller. A concrete open question, for which we have no conjecture, is determining whether the growth of $\mu_c(\delta_2,d)$ is polynomial, exponential, or super-exponential in $d$.

\subsection{Organization}
In Section \ref{sec:ss}, we construct a self-similar threshold frog model and derive some important properties. Section \ref{sec:operator} introduces a more general operator. We relate this operator to the self-similar threshold frog model and then prove a bootstrapping result in \thref{prop:dom}. Section~\ref{sec:proof} uses these results to prove \thref{thm:main}. The appendix contains some useful properties of the Poisson distribution.

 \section{The self-similar threshold frog model} \label{sec:ss}

Here we define the modified frog model described after \thref{thm:main} and deduce some of its properties. Note that we write $\Geo(p)$ for a geometric random variable with parameter $p$ supported on $1,2,\hdots$ and $\Geo_0(p)$ for the analogue supported on $0,1,\hdots$. We will sometimes abuse notation and write $Y \sim \Geo(p)$ to mean that the random variable $Y$ has the indicated distribution.

\subsection{Properties of the Poisson distribution}

We begin by stating two properties of the Poisson distribution that are essential to our later arguments.

\begin{lemma} \thlabel{lem:thinning}
Suppose that $Z$ is a Poisson random variable with mean $\lambda$. Let $Z_1, Z_2, \hdots $ be an independent and identically distributed sequence of random variables with $\P(Z_i = j) = p_j$ for $1\leq j \leq k$ and set $N_j = |\{ m \leq N \colon X_m = j\}|$. Then the $N_1,\hdots, N_k$ are independent and each $N_j$ is a Poisson random variable with mean $p_j \lambda_j$. 
\end{lemma}

\begin{proof}
See \cite[Chapter 3]{durrett2019probability}.
\end{proof}

The second property provides simple criteria for comparing a Poisson distribution with a random parameter to one with deterministic parameter. 

\begin{lemma} \thlabel{lem:compare}
Let $\Theta$ be a nonnegative random variable, $Y \sim \Poi(\Theta)$, and $Z \sim \Poi(\lambda)$ for a parameter $\lambda \geq 0$. The following are equivalent:
\begin{enumerate}[label = (\roman*)]
    \item $Y \succeq Z$.
    \item $\P(Y=0) \leq \P(Z=0)$.
    \item $\E[e^{-\Theta}] \leq e^{-\lambda}$.
    %\item $-\log (\E[e^{-\Theta}]) \geq \lambda $.
\end{enumerate}
\end{lemma}
\begin{proof}
See \cite[Theorem 3.1 (b)]{misra2003stochastic}.
\end{proof}

Lastly, we apply \thref{lem:compare} for the case of a thinned Poisson distribution.

\begin{lemma} \thlabel{lem:geo_compare}
Fix $c>0$, $p\in (0,1)$, and let $X$ be a nonnegative, almost surely finite random variable. Given $\lambda >0$, there exists $\mu_\lambda>0$ such that $\Poi(c \mu d^{-X})\succeq  \Poi(\lambda)$ for all $\mu \geq \mu_\lambda$. 
%with $$\Theta \sim \ind{N=0}\mu   + \ind{N\geq 1} \mu  d^{-1} x^{-N+1}$$ 
\end{lemma}

\begin{proof}
Let $Y \sim \Poi(c \mu d^{-X})$ and $Z\sim \Poi(\lambda)$. 
We will apply the criteria  from \thref{lem:compare} and prove that $\P(Y =0) \leq \P(Z=0) = e^{-\lambda}$ for $\mu$ large enough. We write
\begin{align}
\P(Y=0) &=  \textstyle \sum_{x=0}^\infty  e^{-c \mu d^{-x}} \P(X=x). \label{eq:start}
\end{align}
As this is bounded by $\sum_{x=0}^\infty \P(X=x)$, and each summand decreases to $0$ as $\mu \to \infty$, the dominated convergence theorem ensures that $\P(Y=0) \to 0$ as $\mu \to \infty$. This gives the claimed inequality for $\mu$ large enough.
\end{proof}

\subsection{Construction} \label{sec:construction}
Concisely, in the self-similar frog model active frogs follow lazy non-backtracking random walk paths and at most one active frog is allowed to move away from the root to each vertex, that same frog is the only frog able to activate the frogs there.  We now make this more precise.

We say that a walk or frog is \emph{killed} if the path being followed is terminated. This ensures that a killed frog no longer contributes in any way to the process. A \emph{lazy non-backtracking random walk started at $v \in \TT_d$} begins at $v$ at $t=0$. At $t=1$ it moves to a uniformly random neighbor of $v$. Any subsequent steps at $u \neq v$ the walk remains at $u$ with probability $1/(d+1)$ and otherwise moves to a uniformly random neighbor of $u$ that is not already contained in the path. Each lazy step at $u$ counts towards meeting the threshold at $u$. To preserve a coupling with the usual random walk, we add the additional rule that the walk is killed upon arrival to $\varnothing$ at any time $t\geq 1$. 
%A \emph{lazy  non-backtracking random walk floated by $m$ from $v$} is sampled according to a lazy non-backtracking random walk started at the vertex obtained after taking $m$ simple random walk steps from $v$. So the eventual starting vertex is in the ball of radius $m$ centered at $v$ and sampled according to the measure induced by $m$-steps of a random walk.

\begin{figure}
\begin{tikzpicture}[scale = 1.25]
\draw[thick] (-1,.5) -- (0,0);
\node[below] at (-1,.5) {$\varnothing$};
%\node[below] at (0,0) {$\varnothing'$};
\node[below] at (6,0) {$v$};
\node[right] at (7,0) {$v'$};
%\draw[dashed] (-1,.5) -- (0,1.2);

\node at (3.75,.25) {$\cdots$};
\draw[dashed] (2,0) -- (3,.5);
\draw[dashed] (2,0) -- (3,.7);
\draw[dashed] (2,0) -- (3,.9);
\draw[dashed] (0,0) -- (1,0.5);
\draw[red,fill=blue, color = red] (1,.5) circle (.4ex);
\node at (1.3,.6) {$v_{1,1}$};
\draw[dashed] (0,0) -- (1,0.7);
\draw[dashed] (0,0) -- (1,0.9);
\draw[dashed] (1,0) -- (2,0.5);
\draw[dashed] (1,0) -- (2,0.7);
\draw[dashed] (1,0) -- (2,0.9);
\draw[thick] (0,0) -- (6,0);
\draw[dashed] (5,0) -- (6,.5);
\draw[dashed] (5,0) -- (6,.7);
\draw[dashed] (5,0) -- (6,.9);
\draw[dashed] (6,0) -- (7,.5);
\draw[dashed] (6,0) -- (7,.7);
\draw[dashed] (6,0) -- (7,.9);
\draw[red,fill=blue, color = red] (7,.9) circle (.4ex);
\node at (7.4,1) {$v_{N,d-1}$};

\draw[thick] (6,0) -- (7,0);

\draw[red,fill=blue, color = blue] (6,0) circle (.4ex);
%\node at (6,0) {\textcolor{blue}{$\bullet$}};
   %\draw[|-|](0,-.75) -- ++(6,0) node [midway,fill=white] {$N$};
\end{tikzpicture}
\caption{The spine $S=S(\varnothing,v)$ and nerves $(v_{y,i})$ for $1 \leq y \leq N= N(\varnothing, v)$ and $1 \leq i \leq d-1$. $v_{1,1}$ and $v_{N,d-1}$ are indicated with red circles. The active frog started at $\varnothing$ visits the sites along $S$ until meeting the threshold at $v$. The blue circle represents the $\X_v$ frogs just activated at $v$. Sites between $\varnothing$ and $v$ are inert. The child vertex $v'$ of $v$ is guaranteed to be visited by an active frog. The nerves may only be traversed by frogs woken at $v$ or emerging from other nerves.} \label{fig:spine}
\end{figure}
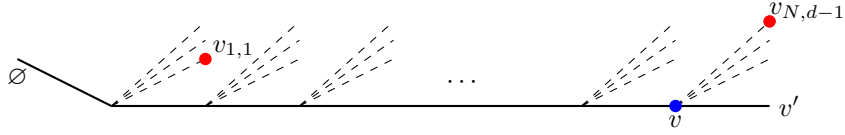

We build the self-similar threshold frog model in an iterative manner. The active frog at the root performs a lazy non-backtracking random walk started at $\varnothing$. Any vertex it visits and fails to reach the threshold at is declared {inert}. The sleeping frogs at inert sites can never be activated. Suppose that $v$ is the first site at which the frog from $\varnothing$ meets the threshold, thus activating the frogs at $v$. Allow the frog that arrived to $v$ to take one more step to a child vertex $v'$ of $v$. Each newly activated frog at $v$ will follow its own independent lazy non-backtracking random walk started at $v$. Let $S(\varnothing, v)$ be the \emph{spine} of vertices on the shortest path between $\varnothing$ and $v'$ (excluding $\varnothing$ and $v'$). Set $N(\varnothing, v) = |S(\varnothing, v)|$ to be the length of the spine. Refer to the edges connected to vertices in $S(\varnothing, v)$ that were not traversed by the frog started at $\varnothing$ as \emph{nerves}. We label them as $(v_{y,i})_{1 \leq y \leq N, 1 \leq i \leq d-1}$ with $N:= N(\varnothing, v)$. See Figure~\ref{fig:spine}.

At this point there is the active frog at $v'$ that originated from $\varnothing$ along with the $\X_v$ frogs just activated at $v$. Now that there may be several active frogs, we introduce the \emph{only-one rule} that whenever a non-backtracking frog moves away from the root to a vertex that has already been visited, that frog is killed. If one or more frogs simultaneously move away from the root to a never-visited site, then one is chosen at random to continue its path and the others are killed. The only-one rule does not apply for non-backtracking steps towards the root. 

Now we explain how the active frogs present at this stage move and wake new frogs. Each follows its lazy non-backtracking random walk and is possibly removed at each step due to the only-one rule. Suppose that a frog, say $f$, meets the threshold at a site $u$ and then moves to a child vertex $u'$ of $u$. Let $w$ be either the most recent site (not equal to $u$) at which $f$ activated the particles at, or, if no such site exists, the starting point of the lazy non-backtracking random walk followed by $f$. For example, if $f$ is the frog started from $\varnothing$ and this is the next iteration, then $w= v$.

With $u, u'$ and $w$ as before, we let $S(w,u)$ be the set of vertices on the shortest path between $w$ and $u'$ (excluding $w$ and $u'$). Set $N(w,u)= |S(w,u)|$ to be the number of vertices on this path. Each newly activated frog begins a lazy non-backtracking walk from $u$. The same rules apply to any newly activated frogs after this step. In this way, the process continues indefinitely. We will refer to this modified process as the \emph{self-similar threshold frog model} $\TSSFM_d(\tau, \mu)$. 

\subsection{Properties} 
Now we turn to proving some useful properties of the self-similar threshold frog model.
The first step is relating $\TSSFM_d(\tau,\mu)$ to the threshold frog model. In what follows we will say that the nonnegative random variable $X$ \emph{stochastically dominates} $Y$, denoted $X \succeq Y$, if there exists a coupling such that $X \geq Y$ almost surely. 
%This is equivalent to having $\P(X \geq a) \geq \P(Y \geq a)$ for all $a \geq 0$. 
We may sometimes write $X \succeq \Poi(\mu)$ to denote that $X$ stochastically dominates a Poisson random variable with mean $\mu$. 

\begin{lemma} \thlabel{lem:ss}
$V \succeq V'$ with $V$ the the total number of root visits in $\TFM_d(\tau, \mu)$ and $V'$ the total number of root visits in $\TSSFM_d(\tau, \mu)$.
\end{lemma}

\begin{proof}
 We rely on the intuitive, often cited monotonicity of the frog model that ignoring parts of the random walk paths and killing frogs decreases the total number of root visits in a given realization of the frog model (see the formal constructions of the frog model in \cite{alves2002phase} or \cite{hermon2018frogs} for more details). The random walk decomposition in \cite[Propostion A.4]{hoffman2019infection} immediately implies that the lazy non-backtracking random walk can be coupled to be a subset of a standard random walk. Since the remainder of the modifications in the self-similar threshold frog model involve killing frogs, it follows that there is a coupling such that the $V' \leq  V$.
\end{proof}

 Next we calculate the probability that the threshold is exceeded at a vertex.

\begin{lemma}\thlabel{lem:exceed}
Suppose that an active frog moves to an unvisited vertex $v$ in the self-similar threshold frog model.
Let $\alpha$ denote the probability that the sleeping frogs at $v$ become active. Recall that $T$ is a generic threshold random variable with distribution $\tau$, and $N$ is the length of the spine between activated vertices.
It holds that 
$$\alpha  = \E [(d+1)^{-T+1}] \text{ and } N \sim \Geo( \alpha).$$
\end{lemma}

\begin{proof}
  Let $G_v \sim \Geo_0(d/(d+1))$. In the self-similar threshold frog model, the frog that moves to $v$ contributes 1 visit immediately plus $G_v$ more visits. So, $\alpha= \P(G_v \geq T_v -1)$. Conditioning on the value of $T_v$ gives claimed equality.
 %$$\P(G_v \geq T_v) = \sum_{k=1}^\infty (d+1)^{-k} \tau(k) = \E (d+1)^{-T}.$$
\end{proof}

We say that a nerve $v_{y,i}$ is \emph{activated} if an active frog moves to $v_{y,i}$ in $\TSSFM_d(\tau, \mu)$. Self-similar frog models earn their name because the number of active frogs that return from activated nerves can be expressed as identically distributed random variables. This gives a recursive distributional equation for the total number of visits to the root. In the following lemma, we abuse notation and write $\Bin(Z,Q)$ to represent the \emph{binomial thinning} random variable that, conditional on $Z=z$ and $Q=q$, has a $\Bin(z,q)$ distribution.

\begin{lemma} \thlabel{lem:rec}
In $\TSSFM_2(\tau, \mu)$, let $S(\varnothing, v)$ be the spine with nerves $(v_{y,i})$ for $1 \leq y \leq N:=N(\varnothing, v)$ and $1 \leq i \leq d-1$. Define the indicator random variables $A_{y,i} = \ind{\text{nerve $v_{y,i}$ is activated}}$.
%and let $(a_y)_{y=0}^N \in \{0,1\}^{N+1}.$ 
Let $V'$ be the total number of root visits in $\TSSFM_2(\tau,\mu)$ and $V_{v'}', V_{1,1}',\hdots V_{N,d-1}'$ be independent and identically distributed copies of $V'$.  It holds that 
\begin{align}
    V'&\overset{d}= \Bin\left(\X_v,d^{-N+1}/(d+1)\right) + \Bin\left(V_{v'}',d^{-N}\right) +  \sum_{y,i} A_{y,i} \Bin\left(V_{y,i}', d^{y-N}\right). \label{eq:rec}
\end{align}
The sum is over all $(y,i)$ in $\{1,\hdots N\}\times \{1,\hdots, d-1\}$. 
\end{lemma}
\begin{proof}
This self-similarity has been observed multiple times starting with \cite[Proposition 6]{hoffman2017recurrence}. See also \cite[Lemma 3.5]{johnson2016critical}. It is a direct consequence of restricting to non-backtracking random walk paths and the only-one rule. The difference with the threshold frog model is that we consider a random number $N$ of nerves. However, this does not impact the way self-similarity is deduced. We remark that the Binomial thinnings use the probability that a non-backtracking active frog that has moved to the spine will reach $\varnothing$. The probability is slightly different for the $\X_v$ activated frogs at $v$ since the first non-backtracking step has $d+1$ rather than $d$ possibilities.
\end{proof}

We will require the fact that, by taking $\mu$ large enough, $V'$ can be made to dominate a large Poisson random variable.

\begin{lemma}\thlabel{lem:start}
Given  $\lambda_0 \geq 0$, there exists $\mu_0 \geq 0$ such that for all $\mu \geq \mu_0$ we have $V' \succeq \Poi( \lambda_0)$.
\end{lemma}

\begin{proof}
It follows from \thref{lem:rec} that $V' \succeq \Bin(\X_v, d^{-N+1} /(d+1))$. Since $\X_v \sim \Poi(\mu)$, Poisson thinning  ensures that $V' \succeq \Poi( (\mu/(d+1)) d^{-N+1} )$. The existence of $\mu_0$ follows immediately from \thref{lem:geo_compare}.
\end{proof}

\section{A self-similar frog model operator} \label{sec:operator}

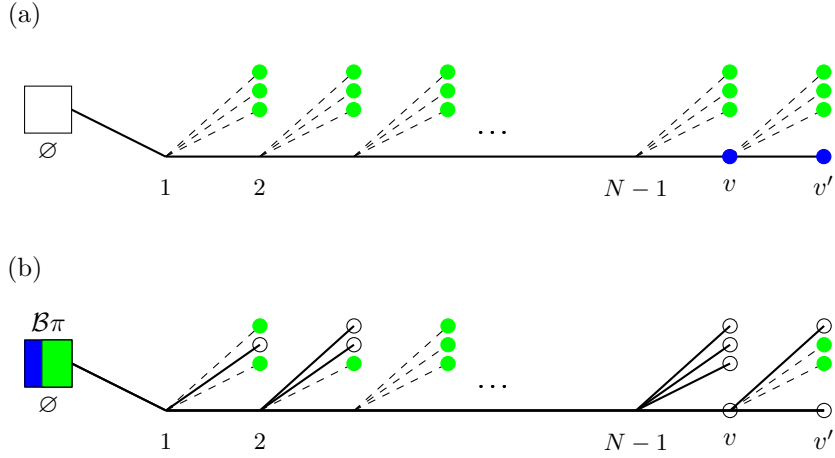
\begin{figure}
\begin{tikzpicture}[scale = 1.25]
\node at (-1.5,1.5) {(a)};
\draw[thick] (-1,.5) -- (0,0);
\draw[thick] (0,0) -- (6,0);
\draw[dashed] (0,0) -- (1,0.7);
\draw[dashed] (0,0) -- (1,0.5);
\draw[dashed] (0,0) -- (1,0.9);
\draw[fill = green, color = green] (1,0.5) circle (.5ex);
\draw[fill = green, color = green] (1,0.7) circle (.5ex);
\draw[fill = green, color = green] (1,0.9) circle (.5ex);
\draw[dashed] (1,0) -- (2,0.5);
\draw[dashed] (1,0) -- (2,0.7);
\draw[dashed] (1,0) -- (2,0.9);
\draw[fill = green, color = green] (2,0.5) circle (.5ex);
\draw[fill = green, color = green] (2,0.7) circle (.5ex);
\draw[fill = green, color = green] (2,0.9) circle (.5ex);
\draw[dashed] (2,0) -- (3,.5);
\draw[dashed] (2,0) -- (3,.7);
\draw[dashed] (2,0) -- (3,.9);
\draw[fill = green, color = green] (3,0.5) circle (.5ex);
\draw[fill = green, color = green] (3,0.7) circle (.5ex);
\draw[fill = green, color = green] (3,0.9) circle (.5ex);
\draw[dashed] (5,0) -- (6,.5);
\draw[dashed] (5,0) -- (6,.7);
\draw[dashed] (5,0) -- (6,.9);
\draw[fill = green, color = green] (6,0.5) circle (.5ex);
\draw[fill = green, color = green] (6,0.7) circle (.5ex);
\draw[fill = green, color = green] (6,0.9) circle (.5ex);
\draw[dashed] (6,0) -- (7,.5);
\draw[dashed] (6,0) -- (7,.7);
\draw[dashed] (6,0) -- (7,.9);
\draw[fill = green, color = green] (7,0.5) circle (.5ex);
\draw[fill = green, color = green] (7,0.7) circle (.5ex);
\draw[fill = green, color = green] (7,0.9) circle (.5ex);
\draw[thick] (6,0) -- (7,0);
\draw[fill=blue, color = blue] (6,0) circle (.5ex);
\draw[fill=blue, color = blue] (7,0) circle (.5ex);
\node[below] at (0,-.15) {\small $1$};
\node[below] at (1,-.15) {\small $2$};
\node[below] at (6,-.15) {$v$};
\node[below] at (7,-.1) {$v'$};
\draw[] (-1,.25) -- (-1,.75) -- (-1.5,.75)  -- (-1.5,.25) -- node[below] {$\varnothing$} (-1,.25)--cycle;
\node at (3.5,.25) {$\cdots$};
\node at (3.5,.25) {$\cdots$};
\node at (3.5,.25) {$\cdots$};
\node[below] at (5,-.15) {\small $N-1$};
\end{tikzpicture}

\vspace{ .5 cm}

\begin{tikzpicture}[scale = 1.25]
\node at (-1.5,1.5) {(b)};
\draw[thick] (-1,.5) -- (0,0);
\draw[thick] (0,0) -- (6,0);
\draw[thick] (6,0) -- (7,0);

\draw[] (-1,.25) -- (-1,.75) -- (-1.5,.75)  -- (-1.5,.25) -- node[below] {$\varnothing$} (-1,.25)--cycle;
\node[above] at (-1.25, .75) {$\B \pi$};
\draw[fill=blue] (-1,.25) -- (-1,.75) -- (-1.5,.75)  -- (-1.5,.25) -- (-1,.25)--cycle;

\draw[fill = green] (-1,.25) -- (-1,.75) -- (-1.32,.75)  -- (-1.32,.25) -- (-1.32,.25)--cycle;
\node[below] at (0,-.15) {\small $1$};
\node[below] at (1,-.15) {\small $2$};
\node[below] at (5,-.15) {\small $N-1$};
\draw[thick] (-1,.5) -- (0,0);
\draw[thick] (0,0) -- (6,0);
\draw[thick] (0,0) -- (1,0.7);
\draw[dashed] (0,0) -- (1,0.5);
\draw[dashed] (0,0) -- (1,0.9);
\draw[fill = green, color = green] (1,0.5) circle (.5ex);
\draw[] (1,0.7) circle (.5ex);
\draw[fill = green, color = green] (1,0.9) circle (.5ex);
\draw[dashed] (1,0) -- (2,0.5);
\draw[thick] (1,0) -- (2,0.7);
\draw[thick] (1,0) -- (2,0.9);
\draw[fill = green, color = green] (2,0.5) circle (.5ex);
\draw[] (2,0.7) circle (.5ex);
\draw[] (2,0.9) circle (.5ex);
\draw[dashed] (2,0) -- (3,.5);
\draw[dashed] (2,0) -- (3,.7);
\draw[dashed] (2,0) -- (3,.9);
\draw[fill = green, color = green] (3,0.5) circle (.5ex);
\draw[fill = green, color = green] (3,0.7) circle (.5ex);
\draw[fill = green, color = green] (3,0.9) circle (.5ex);
\draw[thick] (5,0) -- (6,.5);
\draw[thick] (5,0) -- (6,.7);
\draw[thick] (5,0) -- (6,.9);
\draw[] (6,0.5) circle (.5ex);
\draw[] (6,0.7) circle (.5ex);
\draw[] (6,0.9) circle (.5ex);
\draw[dashed] (6,0) -- (7,.5);
\draw[dashed] (6,0) -- (7,.7);
\draw[thick] (6,0) -- (7,.9);
\draw[fill = green, color = green] (7,0.5) circle (.5ex);
\draw[fill = green, color = green] (7,0.7) circle (.5ex);
\draw[] (7,0.9) circle (.5ex);
\draw[thick] (6,0) -- (7,0);
\draw[] (6,0) circle (.5ex);
\draw[] (7,0) circle (.5ex);
\draw[thick] (0,0) -- (6,0);
\draw[thick] (6,0) -- (7,0);
\node[below] at (6,-.15) {$v$};
\node[below] at (7,-.1) {$v'$};

\node at (3.5,.25) {$\cdots$};
\node at (3.5,.25) {$\cdots$};
\node at (3.5,.25) {$\cdots$};
\end{tikzpicture}
\caption{The initial configuration in (a) has $\Poi(\mu)$ active particles at $v$ and a $\pi$-distributed number of active particles at $v'$ (in blue). The $(d-1)N$ sites $v_{y,i}$ have independent $\pi$-distributed numbers of sleeping particles (in green). Active particles perform non-backtracking random walk and halt upon reaching a leaf. The figure at (b) shows the stabilized state. The operator $\B \pi$ is the law for the number of particles halted at $\varnothing$. We do not depict the active particles halted at $v'$ or any of the visited $v_{y,i}$.} \label{fig:A}
\end{figure}

\thref{lem:rec} gives that $V'$ can be described in terms of which nerves of $S(\varnothing,v)$ are visited. This inspires the following more general version that replaces the i.i.d.\ copies of $V'$ with arbitrary random variables. The construction below is a modification of \cite[Section 2.2]{hoffman2016transience}.

Fix a probability measure $\pi$ on the nonnegative integers. 
Let $N \sim \Geo(\alpha)$ with $\alpha$ as defined in \thref{lem:exceed}. Start with a path of length $N+2$. 
Call the first site of the path $\varnothing$ and the last two sites $v$ and $v'$. Label the vertices in $(\varnothing, v]$ in order as $1,2,\hdots, N$ and form the graph $G$ by attaching vertices $v_{y,1},\hdots, v_{y,d-1}$ to each $y$ in $(\varnothing, v]$. 
See Figure~\ref{fig:A}. Let $\X_v \sim \Poi(\mu)$ and $W_{1,1},...,W_{N,d-1}$ be i.i.d.\ random variables distributed according to $\pi$. 
Place $\X_v$ particles at $v$ and $W_{y,i}$ particles at each $v_{y,i}$. Other vertices start with no particles.
Activated particles perform independent  non-backtracking random walks. When an active particle reaches a leaf of $G$ that is distinct from its starting location, it halts there. 
Initially, only the particles at $v$ and at $v'$ are active, the others are sleeping. When an active particle moves to a site containing sleeping particles, all particles there activate (the threshold is one). Let $A_{y,i}$ be an indicator for the event that the particles at $v_{y,i}$ are eventually activated. We define $\B \pi$ to be the law for the number of particles that halt at $\varnothing$. A restatement of \thref{lem:rec} in these terms is that the law of $V'$ is a fixed point of $\B$.

\begin{lemma}\thlabel{lem:fixed}
Let $\nu$ be the probability measure associated to $V'$ in $\TSSFM_2(\tau,\mu)$. It holds that $\B \nu  = \nu$. 
\end{lemma}

\begin{proof}
The description of $V'$ in \thref{lem:rec} couples in a straightforward way with the definition of the random variable associated to $\mathcal B \nu$. More details can be found in \cite[Lemma 9]{hoffman2016transience}.
\end{proof}

We also note that $\B \pi$ is monotone. 

\begin{lemma}\thlabel{lem:mono}
Say that $\pi \preceq \pi'$ if $\pi([0,x)) \geq \pi'([0,x))$ for all $x \geq 0$. If $\pi \preceq \pi'$, then $\B \pi \preceq \B \pi'$.
\end{lemma}

\begin{proof}
$\B \pi$ is the number of visits to a site in a frog model and thus is monotone when adding additional particles.
\end{proof}

The operator $\B$ is particularly simple when acting on the Poisson distribution.

\begin{lemma} \thlabel{lem:prec}
Let $\lambda \geq 0$. For $\pi \sim \Poi(\lambda)$ it holds that
\begin{align}
    \B \pi \sim \Poi\left( \f{\mu}{d+1} d^{-N+1}  + \lambda d^{-N} - \lambda  \sum_{y,i} A_{y,i} d^{-y}  \right). \label{eq:poi}
\end{align}
The sum is over all $(y,i)$ in $\{1,\hdots N\}\times \{1,\hdots, d-1\}$
\end{lemma}

\begin{proof}
It follows from Poisson thinning (see \thref{lem:thinning}) that $\Bin( \Poi(\lambda), p) \sim \Poi(p \lambda)$. Applying addititivity of the Poisson distribution, and the independence property of Poisson thinning gives \eqref{eq:poi}. 
\end{proof}

\begin{lemma} \thlabel{lem:A}
 Write $\a = (a_{y,i})_{1\leq y \leq N, 1 \leq i \leq d-1 } \in \{0,1\}^{(d-1)n} $ and let $A^\a$ be the event that the set of nerves $\{ v_{y,i} \text{ such that } a_{y,i}=1\}$ all get activated in the variant of the system used to define $\B$ that has no particles at each nerve $v_{x,j}$ with $a_{x,j}=0$ (so $W_{x,j} =0$ for such $x,j$). Set $A = (A_{y,i})$. It holds that
\begin{align}
\P(A = \a \mid N=n)
&= \exp\Bigl( - \sum_{x,j } (1-a_{x,j}) \Bigl[ \tfrac \mu {d+1} d^{x-n}   +  \lambda d^{x-n-1} +  \lambda\sum_{y,i}  a_{y,i} d^{- |y-x|-1} \Bigr] \Bigr) \\
&\hspace{7.25 cm} \times \P(A^\a \mid N=n).
\end{align}
The sum is over all $(x,j)$ and $(y,i)$ in $\{1,\hdots N\}\times \{1,\hdots, d-1\}$
\end{lemma}

\begin{proof}
Conditional on $\A^a$ and $N=n$, in order to have $\A = \a$, none of the nerves $v_{x,j}$ with $a_{x,j} = 0$ are visited by active particles from $v, v'$ and $v_{y,i}$ with $a_{y,i}=1$. Poisson thinning (\thref{lem:thinning}) and additivity ensures that the number of particles moving to each $v_{x,j}$ with $a_{x,j} = 0$ is a Poisson random variable with parameter
$$m_{x,j} = \f \mu {d+1} d^{x-n} +  \lambda d^{x-n-1} + \lambda \sum_{y,i}  a_{y,i}  2^{- |y-x|-1}.$$
Again by Poisson thinning, these Poisson random variables are independent. Writing $\P(A^\a \mid N=n)\prod_{a_{x,j}=0} e^{-m_{x,j} }$ gives the claimed formula.
\end{proof}

\begin{proposition} \thlabel{prop:dom}
There exists $\lambda_0, \mu_0 \geq 0$ such that for all $\lambda \geq \lambda_0$ and $\mu \geq \mu_0$ it holds that 
$V' \succeq \Poi(\lambda)$ implies $V' \succeq \Poi(\lambda + 1)$.
\end{proposition}

\begin{proof}
Fix $\lambda \geq 0$ and assume that $V' \succeq \Poi(\lambda)$. It follows from \thref{lem:fixed}, \thref{lem:mono}, and \thref{lem:prec}  that 
\begin{align}
V ' &\succeq \Poi\left(\textstyle \f{\mu}{d+1} d^{-N+1}  + \lambda d^{-N} + \lambda  \sum_{y,i} A_{y,i} d^{-y}  \right) \label{eq:Vp}.
\end{align}
Let $\Theta$ to be the random Poisson parameter in \eqref{eq:Vp}. \thref{lem:compare} gives that $\Poi(\Theta) \succeq \Poi(\lambda+1)$ if and only if $\E[ e^{-\Theta} ] \leq e^{-\lambda -1}$.  Proving this inequality for $\mu$ and $\lambda$ large enough is our goal. 

To this end, we start by writing
\begin{align}
\E[ e^{-\Theta}] &=  \textstyle \sum_{n=0}^\infty e^{- \f \mu {d+1} d^{-n+1}} \P(N=n) s_n  \label{eq:ETheta}
\end{align}
with
\begin{align}
s_n &:= \sum_{\a} \exp\left( -\lambda ({  d^{-n}+ \textstyle \sum_{y,i} a_{y,i} d^{y-n}  })\right) \P(A =\a \mid N=n ). \label{eq:isum}
    %&=\E \left[\exp\left(- \f \mu 3 2^{-N+1}   + \lambda2^{-N}  - \lambda \left( 2^{-N}+ \textstyle \sum_{A_y=1} 2^{y-N} \right) \right)\right]
\end{align}
Note that the sum in $s_n$ is over all $\a \in \{0,1\}^{(d-1)n}$. 
%Note that several times throughout the remaining argument we will implicitly use the fact that all summands are nonnegative. 
We first require a bound on $s_n$. Let
\begin{align}
h_n(\a) &:=  d^{-n}+  \sum_{y,i} a_{y,i} d^{y-n} + \sum_{x,j} (1-a_{x,j}) \left[  d^{x-n-1} + \sum_{y,i} a_{y,i}  d^{- |y-x|-1} \right], \label{eq:h}
\end{align}
and $g_n(\a):= \f \mu {d+1} d^{x-n}.$
Applying \thref{lem:A} lets us write
\begin{align}
s_n &= \sum_{\a } e^{- \lambda h_n(\a) - g_n(\a) } \P(A^\a \mid N=n).
\end{align}
The trivial bounds $g_n(\a) \geq 0$ and $\P(A^\a \mid N=n ) \leq 1$ give
\begin{align}
	s_n \leq \sum_{\a } e^{-\lambda h_n(\a)}= 2^{(d-1)n} \E [e^{- \lambda h_n(\a)} ]. \label{eq:Sh}
\end{align}
In the last equality, the expectation is taken with respect to the uniform measure on $\{0,1\}^{(d-1)n}$.

We have reduced bounding $s_n$ to estimating the moment generating function of $h_n$ evaluated at a uniformly random binary string $\a$. We  now prove a bound on this quantity as a lemma.
\begin{lemma} \thlabel{lem:hb} If $\a$ is a uniformly random binary string from $\{0,1\}^{(d-1)n}$ and $h_n$ is as defined at \eqref{eq:h}, then
\begin{align}
\E[e^{-\lambda h_n(\a)}] \leq e^{-\lambda}2^{-(d-1)n+1} (1+ e^{-\lambda /d^2} )^{(d-1)n-1}.\label{eq:hb}
\end{align}
\end{lemma}
\begin{proof}
Rearranging and simplifying $h_n(\a)$ gives
%\begin{align}
%2^{-n}+ \textstyle \sum_{x=1}^n \left[  a_x 2^{-z} + (1-a_x) \left[ 2^{z-n-1} + \sum_{y=1}^n a_y 2^{- |z-y|-1} \right]\right]
%\end{align}
\begin{align}
h_n(\a) 
% &= d^{-n}+ \textstyle \sum_{x,j} \left[  a_{x,j} d^{-x} + (1-a_{x,j}) \left[ d^{x-n-1} + \sum_{y,i} a_{y,i} d^{- |x-y|-1} \right]\right] \\
%     &=  d^{-n}+ \textstyle \sum_{x,j} \left[ d^{-x -n -1}
%      +  \sum_{y,i} a_{y,i} d^{- |x-y|-1} \right] \\
%      & \qquad \qquad \qquad \qquad  +\textstyle  \sum_{x,j} a_{x,j} \left[ d^{-x} - d^{x-n-1} - \sum_{y,i} a_{y,i} d^{- |x-y| -1}  \right] \\
%     &= \left[ d^{-n}+ \textstyle \sum_{x,j}  d^{-x -n -1} \right] 
%      +  \left[  \textstyle \sum_{x,j}\sum_{y,i} a_{y,i} d^{- |x-y|-1} \right] \\
%      & \qquad \qquad \qquad \qquad  +\textstyle  \sum_{x,j} a_{x,j} \left[ d^{-x} - d^{x-n-1} - \sum_{y,i} a_{y,i} d^{- |x-y| -1}  \right] \\
%     &=\textstyle \left[ 2^{-n}+ \sum_{x=1}^n 2^{x-n -1} \right] + \left[ \sum_{x=1}^n  \sum_{y=1}^n (a_y - a_ya_x)  2^{- |x-y|-1}  \right]\\
%     & \qquad \qquad \qquad \qquad\qquad \qquad \qquad \qquad +\textstyle \sum_{x=1}^n  a_x\left[   2^{-x}  -  2^{x-n-1} \right]\\
 &=1 +\sum_{x,j}^n  \sum_{y,i}^n (1-a_{x,j}) a_{y,i} d^{- |x-y|-1}  + \sum_{x,i}  a_{x,i}\left[   d^{-x}  -  d^{x-n-1}  \right]. \label{eq:hexpand}
\end{align}
Let $f_n(\a) := |\{(y,i) \colon a_{y,i} \neq a_{y,i+1}\}|$ with the convention that $a_{y,d} = a_{y+1,1}$. The quantity $f_n(\a)$ counts the number of ``bit flips'' reading $\a$ from left to right. 
Each flip  contributes $d^{-2}$ to the double sum in \eqref{eq:hexpand}. So, the double sum is at least $f_n(\a)/d^2$. As the final sum in \eqref{eq:hexpand} is nonnegative, we have
\begin{align}
	h_n(\a) \geq 1 + f_n(\a)/d^{2} . \label{eq:hc}
\end{align}
If $\a$ is sampled uniformly from $\{0,1\}^{(d-1)n}$, then $f_n(\a) \sim \Bin((d-1)n -1,2^{-1})$. Using the moment generating function of a Binomial random variable gives
\begin{align}
\E[e^{-\lambda f_n(\a)/d^2}] &= (2^{-1} + 2^{-1} e^{-\lambda /d^2})^{(d-1)n-1}\\
    &= 2^{-(d-1)n +1} ( 1 + e^{-\lambda /d^2})^{(d-1) n-1}.
\end{align}
This and \eqref{eq:hc} give \eqref{eq:hb}.
% \begin{align}
% \E[e^{-\lambda h_n(\a)}] \leq e^{-\lambda}2^{-(d-1)n+1} (1+ e^{-\lambda /d^2} )^{(d-1)n-1}.\label{eq:hb}
% \end{align}
\end{proof} 
Continuing towards the goal of bounding $\E[e^{-\Theta}]$, it follows from \eqref{eq:Sh} and \thref{lem:hb} that
\begin{align}
	s_n \leq  
	 e^{-\lambda}2(1 + e^{-\lambda /d^2})^{(d-1)n-1}.  \label{eq:Sb}
\end{align}
Applying \eqref{eq:Sb} to \eqref{eq:ETheta} gives
\begin{align}
	\E [e^{-\Theta}] &\leq  e^{-\lambda }  \textstyle \sum_{n=0}^\infty e^{- \f \mu {d+1} d^{-n+1}} 2 (1+e^{-\lambda/d^2})^{(d-1)n-1}\P(N=n). \label{eq:Theta}
\end{align}
The dominated convergence theorem implies that for $\mu$ and $\lambda$ sufficiently large 
\begin{align}
	\textstyle \sum_{n=0}^\infty  e^{- \f \mu {d+1} d^{-n+1}} 2 (1+e^{-\lambda/d^2})^{(d-1)n-1} \P(N=n) \leq e^{-1} .\label{eq:extra}
\end{align} 
Applying \eqref{eq:extra} to \eqref{eq:Theta} gives 
$\E[e^{-\Theta}] \leq e^{-\lambda - 1}$
and thus $V' \succeq \Poi(\lambda +1)$. 
\end{proof}

\section{Proof of \thref{thm:main}}\label{sec:proof}
\begin{proof}
 The lower bound on $\mu_c(\tau,d)$ follows from the fact that $\mu_c(\tau,d) \geq \mu_c(\delta_1,d)$ which is proven to be positive in \cite[Proposition 15]{hoffman2016transience}. Let $\mu_0$ and $\lambda_0$ be as in \thref{prop:dom}. It follows from \thref{lem:start} that for all $\mu \geq \mu_0$ we have $V' \succeq \Poi(\lambda_0)$. Applying \thref{prop:dom} iteratively shows that $V' \succeq \Poi(\lambda_0 + m)$ for all positve integers $m$, and thus $V'=\infty$ almost surely. Since \thref{lem:ss} gives $V \succeq V'$, we conclude that $\mu_c(\tau,d) \leq \mu_0 < \infty$.
\end{proof}

\appendix

\section{Data Availability Statement}
 Data sharing not applicable to this article as no datasets were generated or analysed.

\bibliographystyle{amsalpha}
\bibliography{threshold}

\end{document}